\documentclass[11pt]{amsart}     
\usepackage{graphicx}
\usepackage{times}      
\usepackage{latexsym}   
\usepackage{amssymb}    
\usepackage{amsmath}    
\usepackage{amsbsy}
\usepackage{amsthm}
\usepackage{amsgen}
\usepackage{amsfonts}
\usepackage{array}
\usepackage{epsfig}
\usepackage{hyperref}
\usepackage{color}
\usepackage{verbatim}
\usepackage{psfrag}

\newtheorem{theorem}{Theorem}[section]

\newtheorem*{thm0}{Theorem}

\newtheorem{lemma}[theorem]{Lemma}

\newtheorem{proof-ref}{Proof(Proposition{P:5t2})}

\def\rrr{\mathbb{R}}
\def\ccc{\mathbb{C}}
\def\zzz{\mathbb{Z}}

\def\hh{\mathbb{H}}

\DeclareMathOperator{\Fix}{Fix}
\def\bdm{\begin{displaymath}}
\def\edm{\end{displaymath}}
\def\beq{\begin{equation}}
\def\eeq{\end{equation}}
\def\bes{\begin{equation*}}
\def\ees{\end{equation*}}
\def\epcm{\end{picture}\end{center}\end{minipage}}
\def\bpcm{\begin{minipage}{80pt}\begin{center}\begin{picture}}

\def\t2{T^2}

\def\f4{F_4}
\def\g2{G_2}

\def\p2{\frac{\pi}{2}}
\def\dist{\textrm{dist}}

\def\Fix{\textrm{Fix}}
\def\txt{\textrm}

\def\dim{\textrm{dim}}

\theoremstyle{definition}

\newtheorem{example}[theorem]{Example}

\begin{document}


\title[Positively curved manifolds with maximal symmetry rank]{Positively curved manifolds with maximal symmetry-rank}

\author[Karsten Grove]{Karsten Grove$^{*}$}
\address{University of Maryland\\
 College Park, MD 20742}
\email{kng@math.umd.edu}
\author[Catherine Searle]{Catherine Searle}
\address{Centro de Investigaci\'on y Estudios Avanzados del Instituto Politecnico \\
M\'exico D.F., M\'exico}
\email{csearle@matcuer.unam.mx}

\thanks{$^*$ Supported in part by an NSF grant.}




\begin{abstract}

The symmetry-rank of a riemannian manifold is by definition the rank of its isometry group. We
determine precisely which smooth closed manifolds admit a positively curved metric with maximal
symmetry-rank.

\end{abstract}
\maketitle

\section*{Introduction}
Due to obvious ambiguities there are many interesting aspects to the following
general type problem: {\it Classify positively curved riemannian manifolds with large
isometry groups}. One example of this is of course the well-known classification of
homogeneous manifolds with positive (sectional) curvature (cf. \cite{AW, BB, B, W}. In this
paper, ``large" refers to the rank of the isometry group. Specifically, we define the
symmetry rank of a riemannian manifold $M$ to be

\bdm
\textrm{symrank}(M)=\txt{rank } \textrm{ Iso}(M),
\edm
where $\txt{Iso}(M)$ is the isometry group of $M$. The following solution to the above
problem in this context was inspired by the work of Allday and Halperin in 
\cite{AH} and of
Hsiang and Kleiner in \cite{HK}.

\begin{thm0}
Let $M$ be an n-dimensional, closed, connected riemannian manifold with
positive sectional curvature. Then
\begin{enumerate}
\item [(i)] $\txt{symrank}(M) \leq \lfloor (n + 1)/2 \rfloor$.
\end{enumerate}
Moreover,
\begin{enumerate}
\item [(ii)] Equality holds in $(i)$ only if $M$ is diffeomorphic to either a sphere, a real or complex
projective space, or a lens space.
\end{enumerate}
\end{thm0}

Note that any of the manifolds listed in (ii) above with their standard metrics have
maximal symmetry rank (cf. Example \ref{e:5}).
The main techniques used in the proof of this result are convexity and critical point
theory (cf. \cite{Gr}) applied to the positively curved, singular, orbit spaces. For other
contributions to the general problem we refer to \cite{GS, S1}  and \cite{S2}.
It is our pleasure to thank Stephan Stolz for suggestions leading to diffeomorphism
classification rather than topological classification.

\section{Circle actions with maximal fixed point set}\label{s:1}

In this section we will show how the presence of an effective isometric circle action
with large fixed-point set determines the type of the manifold. First let us recall, that
the fixed-point set $\Fix(S^{1})$ of an isometric $S^{1}$-action is a union of totally geodesic
submanifolds of even codimension. Here we let  dim $\Fix(S^{1})$ denote the largest
dimension of components of $\Fix(S^{1})$. Note also, that the above statements hold as well
for $\Fix(\zzz_{q})$, where $\zzz_{q}$, is any subgroup of $S^{1}$.
From now on we confine our attention to positively curved manifolds $M$. The key
observation is already contained in the following.

\begin{lemma}\label{l:1}
 Suppose codim $\Fix(S^{1}) = 2$. Then:
 \begin{enumerate}
\item [(i)] Exactly one component, $N$, of $\Fix(S^{1})$ has codimension $2$.
\item [(ii)] There is a unique orbit $S^1 p_{0}$ at maximal distance from $N$.
\item [(iii)] $S^{1}$ acts freely on $M - (N \cup S^1  p_{0})$.
\end{enumerate}
 \end{lemma}

 \begin{proof}
Since any two totally geodesic submanifolds $N_{1}$ and $N_{2}$ with 
 $\dim \  N_{1} +
\dim \ N_{2} \geq  \dim \ M$ intersect by a Synge type argument, part (i) is clear.

Now consider the orbit space $X = M/S^{1}$. This is a positively curved space (in
Aleksandrov distance comparison sense, cf. \cite{BGP}), with totally geodesic boundary
$\partial X = N$. As in riemannian geometry (cf. \cite{CG}), $\dist(N,-): M/S^{1} \rightarrow \rrr$ is a strictly
concave function (for the case of general Aleksandrov spaces see \cite{P}). In particular,
there is a unique point $\bar{p}_{0} \in M/S$ (the soul) at maximal distance from $N$. This
proves (ii).

To prove (iii), consider a point $\bar{p} \in X - (N \cup \bar{p}_{0})$. Since $\bar{p}$ is in the boundary of the
convex set $\{ \bar{x} \in X | \txt{  } \dist(N, \bar{x}) \geq \dist(N, \bar{p})\}$, every segment (minimal geodesic) from
$\bar{p}$ to $N$ makes an angle $> \pi/2$ to every segment from $\bar{p}$ to $\bar{p}_{0}$. If $p \in M$ and $S^{1}p ={p}$, the
same statement holds for horizontal lifts of segments to $p$. Clearly, the isotropy $S^{1}_{p}$ at
$p$ will preserve each set of ``opposite" segments. Consequently, by the angle condition
above, $S^{1}_{p}$ will fix a direction pointing toward $N$ (i.e. making an angle $< \pi/2$ to all
segments from $p$ to $N$ and an angle $> \pi/2$ to all segments from $p$ to $S^{1} \cdot p_{0}$). The same
argument, however, also applies to a point $q \in \Fix(S^{1}_{p})$ closest to $N$. Since all points
sufficiently close to $N$ have trivial isotropy, this leads to a contradiction unless $S^{1}_{p}$ is
trivial. 
\end{proof}

The structure obtained in Lemma \ref{l:1} and its proof, provides the ingredients for
establishing our main recognition result.

\begin{theorem}\label{t:2}
 Let $M$ be a closed, connected riemannian manifold with positive curvature.
If $M$ admits an effective isometric $S^{1}$-action with $\textrm{codim}\ \Fix(S^{1}) = 2$, then $M$ is
diffeomorphic to either the unit sphere $S^{n}_{1}$, a space form $S^n_{1}/\zzz_{q}$, or $\ccc P^{m} = S^{n+1}_{1}/S^{1}$ (when
$n = 2m$).
\end{theorem}

\begin{proof}
Let $p_{0}, \bar{p_{0}}$ and $N$ be as in Lemma \ref{l:1}. When $N$ is viewed as a subset of $M/S^{1}$,
we denote it by $\bar{N}$. From the angle condition obtained in Lemma \ref{l:1} it follows that
$\dist(\bar{N}, -)$ and $\dist(\bar{p_{0}}, -)$ (resp. $\dist(N, -)$ and $\dist(S^{1} \cdot p_{0}, -))$ have no critical points in
$M/S^{1} - (\bar{N} \cup \bar{p_{0}}$) (resp. $M - (N \cup S^{1} \cdot p_{0}))$. Choose gradient-like vector fields, $\bar{V}$ and
$V$ on $M/S^{1}$ and $M$ respectively, so that $V$ is a horizontal lift of $\bar{V}$ on
$M - (N \cup S^{1} \cdot p_{0}))$ which is radial near $N$ and $S^{1} \cdot p_{o}$ (cf. e.g. \cite{Gr}). This shows in
particular that $M$ is the union of tubular neighborhoods $D_{\varepsilon}(N)$ and $D_{\varepsilon}(S^{1} \cdot p_{0})$
of $N$ and $S^{1} \cdot p_{0}$, respectively. Moreover, $N \cong \bar{N}$ is diffeomorphic to
$\partial D _{\varepsilon} (\bar{p_{0}}) \cong  \partial D _{\varepsilon} (S^{1} p_{0})/ S^{1} \simeq \partial D^{ \bot} _{\varepsilon} (p_{0})/  S^{1} _{p_{0}}$, where $D^{ \bot} _{\varepsilon} (p_{0})$
 denotes the $\varepsilon$-normal disc at
$p_{0}$ to $S^{1} \cdot p_{0}$. In view of this, it is not surprising that $M$ is determined by the isotropy
group $S^{1}_{p_{0}} $ at  $p_{0}$.

Case 1: $S^{1}_{p_{0}} = \{1\}$. Clearly $(M/S^{1}, \bar{N})$ is diffeomorphic to $(D^{n-1}, S^{n-2})$ and the
quotient map $\pi: M  -  N \rightarrow M/S^{1}  -  \bar{N} \simeq D^{n-1}  -  S^{n-2}$ is a (trivial) principal
circle bundle (cf. Lemma \ref{l:1} (iii)).  Moreover, when restricted to $\partial D_{\varepsilon} (\bar{N}) \subset M$,
$\pi: \partial D_{\varepsilon} (N) \rightarrow  \partial D_{\varepsilon} (\bar{N}) \simeq \bar{N} \cong N(\simeq S^{n-2}) $ is nothing but the normal projection in $M$.
In particular, $D_{\varepsilon}(N) \simeq \partial D_{\varepsilon}(N) \times_{S^{1}} D^{2} $ and 
 hence $M$ is diffeomorphic to $(M -  \txt{int} \ D_{\varepsilon}(N)) \cup   \partial D_{\varepsilon}(N)\times_{S^{1}} D^{2} \cong D^{n-1} \times S^{1} \cup S^{n-2} \times D^{2} \simeq S^{n}.$
 
Case 2: $S^{1}_{p_{0}} = S^{1}$. Note by Lemma \ref{l:1} (iii) that $S^{1}_{p_{0}}$ acts freely on the unit sphere
$S^{n-1} \simeq  \partial D_{\varepsilon}(p_{0})$ at $p_{0}$. In particular, $n = 2m$ and $N \cong \bar{N} \simeq \ccc P^{m-1}$. As in the case
above, $ \pi:  \partial D_{\varepsilon} (N) \rightarrow  \partial D_{\varepsilon} (\bar{N}) \simeq \bar{N} \cong N(\simeq \ccc P^{m-1})$ is the normal projection in $M$ and
$ D_{\varepsilon}(N)$ is diffeomorphic to  $\partial D_{\varepsilon} (N) \times_{S^{1}} D^{2}$. Since $M  -  \txt{int} \ D_{\varepsilon} (N) \simeq D_{\varepsilon}(p_{0}) \simeq D^{2m}$ and
$S^{2m-1} \simeq  \partial D_{\varepsilon} (N) \rightarrow N \simeq \ccc P^{m-1}$ is the Hopf map we conclude that $M$ is diffeomorphic
to $D^{2m} \cup S^{2m-1} \times_{S^{1}} D^{2} \simeq \ccc P^{m}$.

Case 3: $S^{1}_{p_{0}}= \zzz_{q}$. First observe that $S^{1}_{p_{0}}$ acts freely on the unit normal sphere $S^{n-2} \simeq \partial D^{ \bot} _{\varepsilon}(p_{0})$, and consequently $N \cong \bar{N} \simeq S^{n-2}/ \zzz_{q}$. By proceeding as above or
considering the $q$-fold universal cover $\tilde{M} \simeq S^{n} $
  we find that $M$ is diffeomorphic to $\rrr P^{n}$
($q = 2$) or to a lens space $L_{q} \simeq  S^{n}/ \zzz_{q} $ ($q > 
  2$ and $n$ odd). 
\end{proof}  

The main issue in Lemma \ref{l:1} and Theorem \ref{t:2} is that if codim $\Fix(S^{1}) = 2$, then $S^{1}$
acts transitively on the fibers of the unit normal bundle of the maximal component
$N \subset \Fix(S^{1})$. A generalization for a general group $G$ with this property will be explored
in \cite{GS}. The corresponding recognition is coarser and contains more topological types.

\section{Maximal torus actions}\label{s:2}

The purpose of this section is to show that any sufficiently large effective and
isometric torus action on a positively curved manifold contains a circle action with
large fixed-point set. This together with the results obtained in Section \ref{s:1} will establish
the Main Theorem in the Introduction.

We consider even and odd dimensions separately.

\begin{theorem}\label{t:3}
 Let $M$ be a closed, connected and positively curved riemannian manifold
of dimension $2n$. Then any effective and isometric $T^{k}$-action has $k \leq n$. If $k = n$,
$T^{n}$ contains an $S^{1}$ with codim $\Fix(S^{1}) = 2$.
\end{theorem}

\begin{proof}

The proof is by induction on $n$. It is a simple and well-known fact that any
isometric $S^{1}$-action on an even-dimensional manifold with positive curvature has
non-empty fixed-point set (cf. e.g. \cite{K}). In dimension two, any such fixed point is of
course isolated. To complete the induction anchor, suppose $T^{2} = S^{1} \times S^{1}$ acts on $M^{2}$.
If $p$ is fixed by say the first circle factor, it must be fixed also by the second factor. The
induced $T^{2}$-action on the unit circle at $p$ must have $S^{1}$ isotropy. This $S^{1}$ will then act
trivially on $M^{2}$ since its fixed-point set has even codimension. Thus, $T^{2}$ cannot act
effectively on $M$.

Assume by induction that the claim has been established for $n \leq l$, and consider
a $T^{l+1}$-action on $M^{2(l+l)}$. Pick an $S^{1} \subset T^{l+1}$ with $\Fix(S^{1})$ of minimal codimension.
We claim that codim $\Fix(S^{1}) = 2$. If not, codim $N = 2(m + l)$, $m \geq 1$ for any component
$N \subset \Fix(S^{1})$. Now, clearly $T^{l} = T^{l+1}/S^{1}$ acts on $N^{2(l-m)}$. By the induction
hypothesis $T^{l} \times N \rightarrow N$ is ineffective, in fact some $T^{m} \subset T^{l}$ acts trivially on $N$. If
$\pi:T^{l+1} \rightarrow T^{l+1}/S^{1}$ is the quotient map, this implies that $T^{m+1} = \pi^{-1} (T^{m}) \subset T^{l+1}$
fixes all points of $N$. Since codim $N$ = codim $\Fix(S^{1})$ was assumed minimal, the
induced action by $T^{m+1}$ on any unit normal sphere, $S^{2m+1}$, to $N$ is almost free. This,
however, is possible only when $m = 0$ (cf. \cite{GG}).

To prove maximality suppose $T^{l+2}$ acts isometrically on $M^{ 2(l+1)}$. From the above,
some $S^{1} \subset T^{l+2}$ has a fixed-point component $N$ of dimension $2l$, and $T^{l+1} =
T^{l+2}/S^{1}$ acts on $N^{2l}$. By the induction hypothesis some $S^{1} \subset TÕ^{l+2}/S^{1}$ acts trivially
on $N$, and so does $T^{2} = \pi^{-1} (S^{1}) \subset T^{l+1}$. As before, this implies that some
$S^{1} \subset T^{2} \subset T^{l+2}$ acts trivially on $M$. 

\end{proof}

The main difference between even and odd dimensions is related to the existence,
respectively possible non-existence of fixed points for circle actions. Nonetheless, we
have the following analogue to Theorem \ref{t:3}.

\begin{theorem}\label{t:4}
Let $M$ be a closed, connected and positively curved riemannian manifold of
dimension $2n + 1$. Then any eflective and isometric $T^{k+1}$-action has $k \leq n$. If $k = n$,
$T^{n+1}$ contains an $S^{1}$ with codim $\Fix(S^{1}) = 2$.
\end{theorem}

\begin{proof}
First consider the case $n = 1$. We must show that any $T^{2}$ action on $M^{3}$ has
$S^{1}$-isotropy. Note that if some isotropy group $T^{2}_{p}$ is non-trivial, the whole torus $T^{2}$
will act on the individual components of $\Fix\txt{ } T^{2}_{p}$. As we have seen, this yields
$S^{1}$-isotropy. On the other hand, a free isometric $S^{1} \times S^{1}$-action on $M^{3}$ would induce
a free isometric $S^{1}$-action on the positively curved $2$-manifold, $M^{2} = M^{3}/S^{1}$ by
O'Neill's formula (cf. \cite{O'N} or \cite{G}). This is impossible by Theorem \ref{t:3}. Using this fact
as before, we see that any $T^{3}$ action on $M^{3}$ must be at least $S^{1}$-ineffective.

By induction assume that Theorem \ref{t:4} holds for $n \leq l$, and consider an isometric
action $T^{l+2} \times M^{2l+3} \rightarrow M^{2l+3}$. If some $S^{1} \subset T^{l+2}$ has non-empty fixed-point set we
proceed exactly as in the proof of Theorem \ref{t:3}. As in the induction anchor, this is
indeed the case unless $T^{l+2}$ acts freely on $M^{2l+3}$. Such an action would induce a free
$T^{l+1}$-action on a positively curved manifold $M^{2l+2}= M^{2l+3}/S^{1}$, which is impossible
by Theorem \ref{t:3}. The maximality statement is now proved exactly as in Theorem
\ref{t:3}. 
\end{proof}

We conclude by exhibiting maximal torus-actions on the possible model spaces.

\begin{example}\label{e:5}
 Consider the standard action $T^{n+1} \times S^{2n+1} \rightarrow S^{2n+1}$ given in complex
coordinates by $(e^{2\pi i \theta_0}, \hdots, e^{2\pi i \theta_{n}}) \cdot (z_{0}, \hdots, z_n) = (e^{2\pi i \theta_{0}} \cdot z_{0}, \hdots, e^{2\pi i \theta_{n}} \cdot z_{n})$. This also
induces a maximal effective $T^{n+1}$-action on all lens spaces $S^{2n+1}/\zzz_{q}$, as well as an
effective $T^{n} = T^{n+1}/\Delta (S^{1})$-action on $\ccc P^{n} = S^{2n+1}/ \Delta(S^{1})$. To get a maximal effective
action of $T^{n}$ on $S^{2n}$ simply suspend the above action $T^{n} \times S^{2n-1} \rightarrow S^{2n-1}$.

We finally point out that our Main Theorem in dimension 4 is not as general as the
one by Hsiang and Kleiner \cite{HK} (they assume only $S^{1}$-symmetry whereas we assume
$T^{2}$-symmetry). On the other hand our conclusion is stronger (by giving a diffeomorphism
classification rather than a topological classification). It may also be worth
mentioning that an analogous result for $SU(2)^{n}$-actions (where one would add only
$\hh P^n$ to the list) does not hold, since e.g. $SU(2)^3$ acts on the $12$-dimensional Flag
manifold $M^{12}$ \cite{W} of quaternion lines in planes on $\hh^3$ (cf. \cite{S1} and \cite{GS} though).
 \end{example}

\end{document}